\documentclass[10pt]{amsart}
\usepackage{amscd,amssymb}
\usepackage[all]{xy}

\theoremstyle{plain}

\newtheorem{thm}{Theorem}
\newtheorem{lem}[thm]{Lemma}
\newtheorem{cor}[thm]{Corollary}

\newtheorem{pro}[thm]{Proposition}

\theoremstyle{definition}

\numberwithin{thm}{section}

\newcommand{\Hom}{\operatorname{Hom}}
\newcommand{\End}{\operatorname{End}}
\newcommand{\Ext}{\operatorname{Ext}}

\newcommand{\md}{\operatorname{mod}}

\newcommand{\gr}{\operatorname{gr}}
\newcommand{\coh}{\operatorname{coh}}

\newcommand{\Lin}{\operatorname{Lin}\nolimits}

\begin{document}

\title
{A note on sheaves without self-extensions on the projective $n$-space.}  
\author{Dieter Happel}
\address{Fakult\"at f\"ur Mathematik, Technische Universit\"at 
Chemnitz, 09107 Chemnitz, Germany. email: happel@mathematik.tu-chemnitz.de}
\author{Dan Zacharia}
\address{Department of Mathematics, Syracuse University,
Syracuse, NY 13244-0001, USA. email: zacharia@syr.edu}

\thanks{The second author is supported by the NSA grant H98230-11-1-0152.}

\begin{abstract} Let ${\bf P}^n$ be the projective $n-$space over the 
complex numbers. In this note we show that an indecomposable rigid coherent
sheaf on ${\bf P}^n$ has a trivial endomorphism algebra. This generalizes 
a result of Drezet for $n=2.$
\end{abstract}

\maketitle


Let  ${\bf P}^n$ be the projective $n$-space over the complex numbers. 
Recall 
that an {\it exceptional sheaf} is a coherent sheaf $E$ such that 
$\Ext^i(E,E)=0$ for all $i>0$ and $\End E\cong\mathbb C$.  Dr\'ezet 
proved in \cite{D} that if $E$ is an indecomposable sheaf over  ${\bf P}^2$ 
such that $\Ext^1(E,E)=0$ then its endomorphism ring is trivial, and also 
that $\Ext^2(E,E)=0$. Moreover, the sheaf is locally free. Partly motivated
 by this result, we prove in this short note that if $E$ is an indecomposable 
coherent sheaf over the projective $n$-space such that $\Ext^1(E,E)=0$, then 
we automatically get that $\End E\cong\mathbb C$.  The proof involves reducing
 the problem to indecomposable linear modules without self extensions over the 
polynomial algebra.  In the
 second part of this article, we look at the Auslander-Reiten quiver of the 
derived category of coherent sheaves over the projective $n$-space. It is
 known (\cite {MVZ1}) that if $n>1$, then all the components are of type 
$\mathbb ZA_{\infty}$.  Then, using the Bernstein-Gelfand-Gelfand 
correspondence (\cite{BGG}) we prove that each connected component 
contains at most one sheaf. We also show that in this case the sheaf 
lies on the boundary of the component.

Throughout this article, $S$ will denote the polynomial ring in $n+1$ 
variables with coefficients in $\mathbb C$. Its Koszul dual is  the 
exterior algebra in $n+1$ variables $x_0, x_1,\ldots,x_n$ which we 
denote by $R$. It is well known that $R$ is a graded, local, finite 
dimensional algebra over $\mathbb C$.  Moreover, $R$ is a selfinjective  
algebra, that is, the notions of free and of injective $R$-modules coincide.  
Denote by $\Lin R$  and $\Lin S$ the categories of linear $R$-modules 
(linear $S$-modules respectively), that is, of the modules having a 
linear free resolution.  It is well known that both $\Lin R$  and 
$\Lin S$ are closed under extensions and cokernels of monomorphisms.
 We have mutual inverse Koszul dualities:
$$\xymatrix@R=15pt{
    \Lin R
   \ar@/^/[r]^{\mathcal E}  & \Lin S
   \ar@/^/[l]^{\mathcal F}}$$
given by $\mathcal E(M)=\bigoplus_{n\geq 0}\Ext_R^n(M,\mathbb C)$
with the obvious action of $S$ on $\mathcal  E(M)$, and with the inverse
duality $\mathcal F$ defined in a similar way.

Let $\md ^Z R$ be the category of finitely generated graded $R$-modules and graded (degree $0$) homomorphisms, where we denote by $\Hom_R(M,N)_0$ the degree zero homomorphisms from $M$ to $N$. Denote by $\underline\md ^Z R$ the stable category of finitely generated graded $R$-modules. The objects of $\underline\md ^Z R$ are the finitely generated graded $R$-modules, and the morphisms are given by $$\underline\Hom_R(M,N)_0=\Hom_R(M,N)_0/{\mathcal P(M,N)_0}$$ where $\mathcal P(M,N)_0$ is the space of graded homomorphisms from $M$ to $N$ factoring through a free $R$-module. The stable graded category $\underline\md ^Z R$ is a triangulated category where the shift functor is given by the first cosyzygy, that is, every distinguished triangle in the stable category is of the form
$A\rightarrow B\rightarrow C\rightarrow\Omega^{-1}A$ where $\Omega\colon\underline\md ^Z R\rightarrow\underline\md ^Z R$ denotes the syzygy functor.  Note also that the stable category $\underline\md^ZR$ has a Serre duality: 
 $$D\underline\Hom_R(X,\Omega^{-m}Y)_0\cong\underline\Hom_R(Y,\Omega^{m+1}X(n+1))_0$$
for each integer $m$ and all $R$-modules $X$ and $Y$, where $X(n+1)$ is the graded shift of $X$. It is immediate to see that we can derive the well-known Auslander-Reiten formula from the above in the case when $m=1$.

Let $M=\bigoplus M_i$ be a finitely generated graded module over the exterior algebra $R$ and let $\xi$ be a homogeneous element of degree 1 in $R$. Following \cite{BGG}, we can view the left multiplication by $\xi$ on $M$ as inducing a complex of $K$-vector spaces $$L_{\xi}M\colon\  \  \cdots M_{i-1}\stackrel{\cdot\xi}\longrightarrow M_i\stackrel{\cdot\xi}\longrightarrow M_{i+1}\rightarrow\cdots$$  We call the graded module $M$ {\it nice} or {\it proper}, if the homology $H_i(L_{\xi}M)=0$ for each $i\neq 0$ and for each homogeneous element $\xi\ne 0$ of degree 1. It is well-known that a graded module $M$ is free if and only if $H_i(L_{\xi}M)=0$ for each $i\in\mathbb Z$ and for each nonzero homogeneous element $\xi\in R_1$. 

 The derived category $\mathcal D^b(\coh {\bf P}^n)$ of coherent sheaves over the projective 
space is also triangulated and, the BGG correspondence (see \cite{BGG}) $$\Phi\colon\underline\md ^Z R\rightarrow\mathcal D^b(\coh {\bf P}^n)$$  is an exact equivalence of triangulated categories, that assigns to each finitely generated graded $R$-module the equivalence class of a bounded complex of vector bundles on ${\bf P}^n$. If $M=\bigoplus M_i$,  we use also $M_i$ to denote the trivial bundle with finite dimensional fiber $M_i$ and we set:
 $$\Phi(M)\colon\ \ \ \dots\longrightarrow\mathcal O(i)\otimes M_i\stackrel{\delta^i}\longrightarrow\mathcal O(i+1)\otimes M_{i+1}\longrightarrow\dots$$
where the differentials $\delta^i$ are defined in the following way. Let $x_0, x_1,\ldots, x_n$ be a basis of $V$, and let $\xi_0, \xi_1,\ldots, \xi_n$ denote the dual basis of $V^*$. Then $\delta^i=\sum_{j=0}^n\xi_j\otimes x_j$. Moreover for each nice $R$-module $M$, the complex $\Phi(M)$ is quasi-isomorphic to the stalk complex of a vector bundle over the projective $n$-space, and all the vector bundles can be obtained in this way from the nice $R$-modules. 
 
Throughout this article if $M$ is a graded module, then its graded shift, denoted by $M(i)$, is the graded module given by $M(i)_n=M_{i+n}$ for all integers $n$. If $M$ is a graded module, we denote its truncation at the $k$-th level by $M_{\geq k}=M_k\oplus M_{k+1}\oplus\cdots$ 

\section{The endomorphism ring of a rigid sheaf.}

\medskip
Let $F$ be an indecomposable coherent sheaf over the projective $n$-space. Assume also that $F$ is {\it rigid}, that is $\Ext^1(F,F)=0$. We prove in this section that the endomorphism ring of $F$ is trivial. This is done by reducing the problem to indecomposable rigid linear modules over the polynomial algebra $S$ in $n+1$ variables.  

Let $M$ be a finitely generated, graded, indecomposable non projective $S$-module. We say that $M$ has no graded self-extensions, if every short exact sequence of graded $S$-modules $0\rightarrow M\rightarrow X\rightarrow M\rightarrow 0$ splits, that is $\Ext_S^1(M,M)_0=0$. 

In order to prove our next result we recall the following definition and facts from \cite{GMRSZ}.  We say that a module $M\in\md^ZS$ has a {\it linear presentation} if it has a graded free presentation 
$P_1\rightarrow P_0\rightarrow M\rightarrow 0$ where $P_0$ is generated in degree 0 and $P_1$ is generated in degree 1. Obviously such a linear presentation is always minimal.  A stronger notion is that of a {\it linear} module. We say that a graded module $M$ is linear, if it has a minimal graded free resolution:
$$\cdots P_k\rightarrow\cdots P_1\rightarrow P_0\rightarrow M\rightarrow 0$$
where for each $i\ge 0$, the free module $P_i$ is generated in degree $i$. 
 If we denote by $\mathcal L_S$ the subcategory of $\md^ZS$ consisting of the modules with a linear presentation, then we have an exact equivalence of categories $\mathcal L_S\cong\mathcal L_{S/{J^2}}$ where $J$ denotes the maximal homogeneous ideal of $S$,  (\cite{GMRSZ}). This equivalence is given by assigning to each module $M$ having a linear presentation over $S$, the module $M/{J^2M}$. We start with the following background result:

\begin{lem} An exact sequence $0\rightarrow M\rightarrow E\rightarrow N\rightarrow 0$ in $\mathcal L_S$ 
splits if and only if the induced sequence $0\rightarrow M/{J^2M}\rightarrow E/{J^2E}\rightarrow N/{J^2N}\rightarrow 0$ splits too.
\end{lem}
\begin{proof} To prove this, it is enough to show that a monomorphism $M\rightarrow E$ in $\mathcal L_S$ splits if and only if the induced map $M/{J^2M}\rightarrow E/{J^2E}$ is also a splittable monomorphism, and one direction is trivial. For the other direction, recall that the category $\mathcal L_S$ is equivalent to the category $\mathcal C$ whose objects are graded morphisms $f\colon P_1\rightarrow P_0$ between free $S$-modules, where $P_0$ is generated in degree zero, $P_1$ is generated in degree one, and the degree one component of $f$ is a monomorphism from the degree one component of $P_1$ to the degree one component of $P_0$. A morphism in $\mathcal C$ from $f\colon P_1\rightarrow P_0$ to $g\colon Q_1\rightarrow Q_0$ is a pair 
$(u_1,u_0)$ of graded homomorphisms such that the diagram 
 $$ 
\xymatrix
{ &P_1 \ar^{u_1}[d]
\ar^{f}[r] &P_0\ar^{u_0}[d]\\ &Q_1 \ar^{g}[r]
&Q_0} 
$$
is commutative. Let $M$ and $E$ be two modules with linear presentations $P_1\stackrel{f}\rightarrow P_0$ and 
$Q_1\stackrel{g}\rightarrow Q_0$, and let $h\colon M\rightarrow E$ be a homomorphismm such that the induced map $\tilde h\colon\ M/{J^2M}\rightarrow E/{J^2E}$ is a splittable monomorphism.  We have an exact commutative diagram:
$$ 
\xymatrix
{ &P_1/{J^2P_1} \ar^{\tilde h_1}[d]
\ar^{\tilde f}[r] &P_0/{J^2P_0}\ar^{\tilde h_0}[d]\ar^{\tilde\pi_M}[r]
&M/{J^2M}\ar^{\tilde h}[d] \ar[r] &0\\ &Q_1/{J^2Q_1}\ar^{\tilde q_1}@{.>}[d] \ar^{\tilde g}[r]
&Q_0/{J^2Q_0}\ar^{\tilde q_0}@{.>}[d]\ar^{\tilde\pi_E}[r]&E/{J^2E}\ar^{\tilde q}[d]\ar[r]& 0\\ &P_1/{J^2P_1}
\ar^{\tilde f}[r] &P_0/{J^2P_0}\ar^{\tilde\pi_M}[r]
&M/{J^2M} \ar[r] &0} 
$$
where $\tilde q\tilde h=1_{M/{J^2M}}$. Observe that the liftings $\tilde q_1$ and $\tilde q_0$ also have the property that $\tilde q_1\tilde h_1=1_{P_1/{J^2P_1}}$ and $\tilde q_0\tilde h_0=1_{P_0/{J^2P_0}}$. By lifting to $\mathcal L_S$ we obtain the following exact commutative diagram:
$$ 
\xymatrix
{ &P_1 \ar^{h_1}[d]
\ar^{f}[r] &P_0\ar^{h_0}[d]\ar^{\pi_M}[r]
&M\ar^{h}[d] \ar[r] &0\\ &Q_1\ar^{q_1}[d] \ar^{g}[r]
&Q_0\ar^{q_0}[d]\ar^{\pi_E}[r]&E\ar[r]& 0\\ &P_1
\ar^{f}[r] &P_0\ar^{\pi_M}[r]
&M \ar[r] &0} 
$$
and we have that $q_1h_1=1_{P_1}$ and $q_0h_0=1_{P_0}$. There exists an unique homomorphism $l\colon E\rightarrow M$ such that $l\pi_E=\pi_Mq_0$, and it follows immediately that $lh=1_M$. 
\end{proof}

We will also need the following results about graded modules over the polynomial algebra $S=\mathbb C[x_0,\ldots x_n]$:

\begin{lem} Let $0\rightarrow A\stackrel{f}\rightarrow B\stackrel{g}\rightarrow C\rightarrow 0$ be an exact sequence in $\md^{\mathbb Z}S$ where $A$ is a linear module and $C$ is semisimple generated in a single degree $m\ge 0$. Then the sequence splits.
\begin{proof} It suffices to show that the sequence $0\rightarrow A_{\ge m}\stackrel{f_{\ge m}}\longrightarrow B_{\ge m}\stackrel{g_{\ge m}}\longrightarrow C\rightarrow 0$ splits in the category of graded $S$-modules since we can easily construct a right inverse to $g$ from a right inverse to $g_{\ge m}$. Each of the modules $A_{\ge m}, B_{\ge m}, C$ is an $m$-shift of a linear $S$-module, so by Koszul duality we obtain an exact sequence over the exterior algebra $R$:

$$0\rightarrow\mathcal E(C)\rightarrow\mathcal E( B_{\ge m})\rightarrow \mathcal E( A_{\ge m})\rightarrow 0$$
\noindent But $R$ is selfinjective and $\mathcal E(C)$ is free, so this sequence splits. The lemma follows now immediately.
\end{proof}
\end{lem}

 \begin{lem} Let $C$ be a linear $S$-module. Let $i$ be a positive integer and assume that  $\Ext_S^1(C_{\geq i},C)_0=0$. Then $\Ext_S^1(C_{\geq i},C_{\geq i})_0=0$.
\end{lem}
\begin{proof} Assume to the contrary that $\Ext_S^1(C_{\geq i},C_{\geq i})_0\ne 0$. We have the following pushout diagram of graded modules: 
$$ 
\xymatrix
{ &0\ar[r]&C_{\geq i} \ar^{j}[d]
\ar^{\alpha}[r] &X\ar^{h}[d]\ar^{\beta}[r]
&C_{\geq i}\ar@{=}[d] \ar[r] &0\\ &0 \ar[r] &C \ar^{\gamma}[r]
&Y\ar^{\delta}[r] &C_{\geq i}\ar[r]& 0}
$$
where $j$ is the inclusion map. The bottom sequence splits by assumption so we have homomorphisms $q\colon C_{\geq i}\rightarrow Y$ and $p\colon Y\rightarrow C$ such that $\delta q=1$ and $p\gamma=1$. But then each graded component of $\gamma,\delta$ splits so we obtain an induced split exact sequence $0\rightarrow C_{\geq i}\rightarrow Y_{\geq i}\rightarrow C_{\geq i}\rightarrow 0$ and an induced commutative diagram:
$$ 
\xymatrix
{ &0\ar[r]&C_{\geq i} \ar@{=}[d]
\ar^{\alpha}[r] &X\ar^{h}[d]\ar^{\beta}[r]
&C_{\geq i}\ar@{=}[d] \ar[r] &0\\ &0 \ar[r] &C_{\geq i} \ar^{\gamma_{\geq i}}[r]
&Y_{\geq i}\ar^{\delta_{\geq i}}[r] &C_{\geq i}\ar[r]& 0} 
$$
and since the two extensions represent the same element in $\Ext^1$ they both split.
\end{proof}

The following result holds for an arbitrary Koszul algebra over the complex numbers.
\begin{lem}  Let $\Gamma$ be a Koszul algebra and let $M\in\mathcal L_\Gamma$ be an indecomposable non projective 
$\Gamma$-module of Loewy length two having no graded self-extensions. Then $\End_{\Gamma}(M)_0\cong\mathbb C$.
\end{lem}
\begin{proof} The proof follows the same idea as the one given in \cite{HR}, but we present it here for the reader's convenience.  Assume that there exists a non zero graded homomorphism $f\colon M\rightarrow M$ that is not an isomorphism, and let $N$ denote its image and $L$ its kernel. Furthermore, let $T$ denote the cokernel of the inclusion map $N\rightarrow M$.  Note that  $T$ is also generated in degree zero, and that $L$ lives in degree 1 and possibly also in degree zero. Applying $\Hom_{\Gamma}(T,-)$ to the exact sequence $0\rightarrow L\rightarrow M\rightarrow N\rightarrow 0$ and taking only the degree zero parts we obtain we get an exact sequence $$\rightarrow\Ext_{\Gamma}^1(T,M)_0\rightarrow\Ext_{\Gamma}^1(T,N)_0\rightarrow\Ext_{\Gamma}^2(T,L)_0.$$
\noindent But if $$\cdots\rightarrow P_2\rightarrow P_1\rightarrow P_0\rightarrow T\rightarrow 0$$ is the beginning of a graded minimal projective resolution of $T$, the second term $P_2$ is generated in degrees $2$ or higher, hence $\Hom_{\Gamma}(P_2,L)_0=0$ and therefore $\Ext_{\Gamma}^2(T,L)_0$ vanishes too. We obtain a surjection $\Ext_{\Gamma}^1(T,M)_0\rightarrow\Ext_{\Gamma}^1(T,N)_0$, and hence a commutative exact diagram of graded $\Gamma$-modules of Loewy length $2$ and degree zero maps:
 $$ 
\xymatrix
{ &0\ar[r]&M \ar^{f}[d]
\ar[r] &X\ar[d]\ar[r]
&T\ar@{=}[d] \ar[r] &0\\ &0 \ar[r] &N \ar[r]
&M\ar[r] &T\ar[r]& 0} 
$$
This pushout yields an exact sequence of graded $\Gamma$-modules $$0\rightarrow M\rightarrow N\oplus X\rightarrow M\rightarrow 0$$ Our assumption on $M$ implies that this sequence must split. The Krull-Remak-Schmidt theorem yields a contradiction, since  $N\ne 0$, $M$ is indecomposable and $\ell(N)<\ell(M)$.
\end{proof}

We can proceed now with the proof of our main result:
\begin{thm}  Let $F$ be an indecomposable coherent sheaf over the projective space ${\bf P}^n$ such that $\Ext^1(F,F)=0$. Then, $\End F\cong\mathbb C$.
\end{thm}
\begin{proof} By Serre's theorem, we may identify the category of coherent sheaves over ${\bf P}^n$ with the quotient category $Q\md^ZS$ of $\md^ZS$. Recall that the objects of $Q\md^ZS$ are the objects of $\md^ZS$ modulo the graded modules of finite dimension. Let $\widetilde X$ denote the sheafification of $X$, that is the image of $X$ in the quotient category. If $X$ and $Y$ are two graded $S$-modules, then $\widetilde X\cong\widetilde Y$ in $Q\md^ZS$ if and only if for some integer $k$, their truncations $X_{\geq k}$ and $Y_{\geq k}$ are isomorphic as graded $S$-modules. 
Now, let $F$ be an indecomposable sheaf with $\Ext^1(F,F)=0$. We may assume that $F=\widetilde X$ where $X$ is a linear $S$-module up to some shift (\cite{AE}), and that $X$ has no finite dimensional submodules. Then (see \cite{MV} for instance),
$\End (F)\cong\varinjlim\Hom_S(X_{\ge k},X)_0$ and $\Ext^1(F,F)\cong\varinjlim\Ext^1_S(X_{\ge k},X)_0$. For each $k$, we have an exact sequence of graded $S$-modules
$$0\rightarrow X_{\ge k+1}\rightarrow X_{\ge k}\rightarrow S_k\rightarrow0$$
\noindent where $S_k$ denotes the semisimple module $X_{\ge k}/{X_{\ge k+1}}$ concentrated in degree $k$. Applying $\Hom_S(-,X)$ and taking degrees we obtain a long exact sequence 
$$0\rightarrow\Hom_S(X_{\ge k},X)_0\rightarrow\Hom_S(X_{\ge{k+1}},X)_0\rightarrow\Ext^1_S(S_k,X)_0\rightarrow\Ext^1_S(X_{\ge k},X)_0\rightarrow\Ext^1_S(X_{\ge{k+1}},X)_0$$
\noindent since $\Hom_S(S_k,X)_0=0$ by our assumption on $X$. $\Ext^1_S(S_k,X)_0$ also vanishes by Lemma 1.2., since having a nontrivial extension $0\rightarrow X\rightarrow Y\rightarrow S_k\rightarrow 0$ yields a nonsplit exact sequence of degree $k$ shifts of linear modules 
$0\rightarrow X_{\ge k}\rightarrow Y_{\ge k}\rightarrow S_k\rightarrow 0$. Since for each $k$, $\Hom_S(X_{\ge k},X)_0=\End_S(X_{\ge k})$, we see that we have $\End(F)\cong\End_S(X)_0\cong\End_S(X_{\ge k})_0$ for each $k\ge 0$. At the same time we have a sequence of embeddings $$\Ext^1_S(X,X)_0\hookrightarrow\cdots\hookrightarrow\Ext^1_S(X_{\ge{k}},X)_0\hookrightarrow\Ext^1_S(X_{\ge{k+1}},X)_0\hookrightarrow\cdots$$    
\noindent Thus, the assumption that $\Ext^1(F,F)=0$ implies that for each $k$ we must have $\Ext_S^1( X_{\ge{k}},X)_0=0$, in particular $\Ext^1_S(X,X)_0=0$.  $X$ is a shift of a linear module so we may apply Lemma 1.1. We now apply 1.4. to conclude that the sheaf $F$ has trivial endomorphism ring. 
\end{proof}

\section{Auslander-Reiten components containing sheaves}

We start with the following general observation: 

\begin{lem} Let $A\stackrel{{\left[\begin{smallmatrix}u_1\\ u_2\end{smallmatrix}\right]}}\longrightarrow B_1\oplus B_2\stackrel{\left[\begin{smallmatrix}v_1\ v_2\end{smallmatrix}\right]}\longrightarrow C\rightarrow A[1]$ be a triangle in a Krull-Schmidt triangulated category $\mathcal T$, and assume that $C$ is indecomposable and both maps $v_1,v_2$ are nonzero. Then $\Hom_{\mathcal T}(A,C)\neq 0$.
\end{lem}
\begin{proof} We prove that $v_1u_1\neq 0$. Assume that $v_1u_1=0$. Then, since the composition ${\left[\begin{matrix}v_1\  0\end{matrix}\right]}{\left[\begin{matrix}u_1\\ u_2\end{matrix}\right]}\colon A\rightarrow C$ is equal to $0$, we have an induced homomorphism $\psi\colon C\rightarrow C$ such that $\psi{\left[\begin{matrix}v_1\ v_2\end{matrix}\right]}={\left[\begin{matrix}v_1\ 0\end{matrix}\right]}$. Thus $\psi v_1=v_1$ and $\psi v_2=0$. Since $C$ is indecomposable, $\psi$ is either an isomorphism or it is nilpotent. It cannot be nilpotent since $v_1\neq 0$ and it cannot be an isomorphism since $v_2\neq 0$, hence a contradiction.
\end{proof}

The BGG correspondence $\Phi$ is an exact equivalence of triangulated categories, hence for every graded $R$-module $M$,
we have that $\Phi(\Omega M)=\Phi(M)[-1]$. An easy computation also shows that for each finitely generated graded module $M$ and integer $i$, we have that 
$\Phi(M(i))\otimes\mathcal O(i)=\Phi(M)[i]$. It turns out that the modules corresponding to sheaves are located on the boundary of the Auslander-Reiten component containing them. This and more follows immediately from the next result by playing with the BGG correspondence. 

\begin{pro} Let $B$ and $C$ be two  $R$-modules corresponding to sheaves under the BGG correspondence. Then, for each $i>0$ $\underline\Hom_R(\tau^iB,C)_0=0$.
\end{pro}
\begin{proof} Since taking syzygies and graded shifts are self-equivalences of $\underline\md ^Z R$, we have by \cite{MVZ1} that $C$ is a weakly Koszul module, and so we can use the formula $\tau B=\Omega^2B(n+1)$, and so $\tau^i B=\Omega^{2i}B(ni+i)$ for all $i\ge 1$.  We have then  $$\Phi(\tau^iB)=\Phi(\Omega^{2i}B(ni+i))=\Phi(\Omega^{2i}B)[ni+i]\otimes\mathcal O(-ni-i)=\Phi(B)[ni-i]\otimes\mathcal O(-ni-i).$$
Applying the BGG correspondence, we obtain:
\begin{align*}
\underline\Hom_R(\tau^iB,C)_0&\cong\Hom_{\mathcal D^b(\coh({\bf P}^n))}(\Phi(\tau^iB),\Phi(C))\\&=\Hom_{\mathcal D^b(\coh({\bf P}^n))}(\Phi(B)[ni-i]\otimes\mathcal O(-ni-i),\Phi(C))\\&\cong\Hom_{\mathcal D^b(\coh({\bf P}^n))}(\Phi(B)(-ni-i)[ni-i],\Phi(C))=0
\end{align*}
where the last equality holds since both $\Phi(C)$ and $\Phi(B)(-ni-i)$ are sheaves, $n\geq 2$, so $ni-i\ge 1$.
\end{proof}

We have the following immediate consequence by letting $B=C$ and $i=1$ in the previous proposition, and applying Lemma 2.1.:
\begin{cor} Let $C$ be an indecomposable $R$-module corresponding to a sheaf under the BGG correspondence, where $n\ge2$. Then the Auslander-Reiten sequence ending at $C$ has indecomposable middle term. \qed
\end{cor}

We end with the following observation:

\begin{pro} Let $\mathcal C$ be a connected component of the Auslander-Reiten quiver of $\mathcal D^b(\coh {\bf P}^n)$ for $n>1$. Then $\mathcal C$ contains at most one indecomposable coherent sheaf.
\end{pro}
\begin{proof} By the previous results, if $\mathcal C$ contains a sheaf, then this sheaf must lie on the its boundary. Under the BGG correspondence, two sheaves in $\mathcal C$ must correspond to two modules of the form $C$ and $\tau^iC$ for some $i\ge 1$. But by applying the BGG correspondence we get stalk complexes concentrated in different degrees, hence we obtain a contradiction. 
\end{proof}



\begin{thebibliography}{GGGGG}

\bibitem{AE} Avramov, L., Eisenbud, D, {\it
Regularity of modules over a Koszul algebra}, Journal of Algebra {\bf
153} (1992), 85-90.

\bibitem{ ARS} Auslander, M., Reiten, I., Smal\o, S. O.;
{\it Representation Theory of Artin Algebras}, Cambridge Studies
in Advanced Mathematics, {\bf  36}, Cambridge University Press,
Cambridge, (1995).

\bibitem {BGG} Bernstein, I.N., Gelfand, I.M., Gelfand, S.I.;
{\it Algebraic bundles over ${\bf P}^n$ and problems in linear algebra},
Funct. Anal. Appl., {\bf 12}, (1979), 212-214.

\bibitem{D} Dr\'ezet, J.-M., {\it Fibr\'es exceptionnels et suite spectrale de Beilinson g\'en\'eralis\'ee sur $P\sb 2(C)$}. 
Math. Ann.  {\bf 275}  (1986),  no. 1, 25--48. 

\bibitem{EG} Evans, E. G., Griffith, P.; {\it The syzygy
 problem,} Annals of Mathematics (2) {\bf 114} (1981), no 2, 323-333.

\bibitem{GMRSZ} Green, E. L., Mart\'inez-Villa, R., Reiten, I.,
Solberg, \O., Zacharia, D.; {\it On modules with linear
presentations}, J. Algebra {\bf  205}, (1998), no. 2, 578--604.

\bibitem{Ha} Happel, D., {\it Triangulated categories in the representation theory of finite-dimensional algebras}. 
London Mathematical Society Lecture Note Series, {\bf 119}. Cambridge University Press, Cambridge, 1988. 

\bibitem{HR} Happel, D., Ringel, C,M,; {\it Tilted Algebras}, 
Transactions of the American Mathematical Society, {\bf 274}, No. 2, 1982, pp. 399-443 

\bibitem{Hu} Huybrechts, D.; {\it Fourier-Mukai Transforms in Algebraic Geometry}, Oxford Mathematical Monographs. Clarendon Press, 2006.

\bibitem{MV} Mart\'inez-Villa, R.; {\it Koszul algebras and sheaves over projective space} arXiv math.RT/0405538.

\bibitem{MVZ1} Martinez-Villa, R.; Zacharia, D. {\it
Approximations with modules having linear resolutions}. J. Algebra
{\bf 266} (2003) 671-697.

\bibitem{MVZ2} Martinez-Villa, R.; Zacharia, D. {\it Auslander-Reiten sequences, locally free sheaves and Chebysheff polynomials.}  Compos. Math.  {\bf 142}  (2006),  no. 2, 397--408.


\end{thebibliography}
\end{document}